\documentclass[11pt,letterpaper]{amsart}

\usepackage{graphicx, psfrag}
\usepackage{amsmath, amscd, amssymb}
\usepackage{graphpap, color}
\usepackage[mathscr]{eucal}
\usepackage{cancel}
\usepackage{verbatim}
\usepackage{adjustbox}
\usepackage{tikz}
\usepackage{float}
\usepackage{booktabs}
\usepackage{multirow}
\usepackage{tabularx}
\usepackage{bbm}
\usepackage{stmaryrd}
\usepackage{mathtools}
\usepackage{subcaption}

\usepackage{listings}

\numberwithin{equation}{section}

\newtheorem{theorem}{Theorem}[section]
\newtheorem{corollary}[theorem]{Corollary}

\newtheorem{conjecture}[theorem]{Conjecture}

\theoremstyle{definition}
\newtheorem{definition}[theorem]{Definition}

\newtheorem{remark}[theorem]{Remark}

\newtheorem{ques}[theorem]{Question}

\newcommand{\PP}{\mathbb{P}}

\def\CC{\mathbb{C}}

\def\EE{\mathbb{E}}

\def\QQ{\mathbb{Q}}

\def\RR{\mathbb{R}}

\def\ZZ{\mathbb{Z}}

\def\sT{\mathscr{T}}

\newcommand{\cal}{\mathcal}

\def\cM{{\cal M}}
\def\cN{{\cal N}}

\def\cS{{\cal S}}

\def\hbar{\overline{h}}

\def\ocM{\overline{\cM}}
\def\Mbar{\overline{\cM}}
\def\mzn{\overline{\cM}_{0,n} }

\def\mgn{\overline{\cM}_{g,n}}

\def\dim{\mathrm{dim} }

\def\log{\mathrm{log} }

\def\and{\quad{\rm and}\quad}
\def\lra{\longrightarrow }

\def\beq{\begin{equation}}
\def\eeq{\end{equation}}
\def\ben{\begin{enumerate}}
\def\een{\end{enumerate}}

\def\DM{Deligne-Mumford }

\def\and{\quad\text{and}\quad}

\title{Cohomology of moduli spaces of pointed curves}

\author{Jinwon Choi}
\address{Department of Mathematics and Research Institute of Natural Sciences, Sookmyung Women's University, Seoul 04310, Korea}
\email{jwchoi@sookmyung.ac.kr}

\author{Young-Hoon Kiem}
\address{School of Mathematics, Korea Institute for Advanced Study, 85 Hoegiro, Dongdaemun-gu, Seoul 02455, Korea}
\email{kiem@kias.re.kr}

\date{}

\begin{document}

\begin{abstract}
    In this paper, after reviewing recent progress on the cohomology of $\mzn$, we further our investigation on the cohomology of moduli spaces of pointed curves in continuation of \cite{BK, CK0, CK, CKL1, CKL2, CKL3}. In particular, we prove that the Betti number distribution of the Fulton-MacPherson compactification $C[n]$ of the space of $n$ ordered distinct points on any smooth projective curve $C$ is asymptotically Gaussian as $n$ goes to infinity.
\end{abstract}
\maketitle

\section{Introduction}\label{S1}

The Deligne-Mumford moduli space of stable pointed curves $\mgn$ is one of the most fundamental objects in algebraic geometry \cite{DM}. It parameterizes a connected projective curve $C$ of arithmetic genus $g$ with $n$ ordered points $p_1,\cdots,p_n$ in $C$ satisfying the stability conditions:
\begin{enumerate}
    \item[(i)] $C$ has at worst nodal singularities and $p_i$ are smooth distinct points;
    \item[(ii)] $\omega_C^\log=\omega_C(p_1+\cdots+p_n)$ is an ample line bundle.
\end{enumerate}
It is well known that $\mgn$ is a smooth irreducible proper \DM stack of finite type and admits a projective coarse moduli space.
It serves as the model for moduli spaces of higher dimensional stable varieties and plays a major role in enumerative geometry and related areas.
Naturally, much efforts have been exerted for more than fifty years to understand the geometry and topology of $\mgn$ but we still do not have satisfactory answers to immediate questions like the following.
\begin{ques}\label{0}
\begin{enumerate}
\item What are the Betti numbers $\dim\, H^k(\mgn)$?
\item Can we describe $H^k(\mgn)$ as an $S_n$-representation where the symmetric group $S_n$ acts on $\mgn$ by permuting the marked points?
\item Does the sequence $\{\dim\, H^k(\mgn)\}_{0\le k\le 6g-6+2n}$ of Betti numbers satisfy interesting properties like log-concavity?
\end{enumerate}
\end{ques}
For $g=0$, the cohomology is much better understood than higher genus cases. We know a recursive formula \eqref{8} for the Betti numbers and explicit ring generators \eqref{10} as well as relations since 1990s from \cite{Getzler, Keel, Manin}. Recently, a new inductive construction of $\mzn$ was discovered in \cite{CK0} which led to  closed and inductive formulae for the $S_n$-representation $H^k(\mzn)$ in \cite{CKL1, CKL2}. Moreover, the even degree Betti numbers of $\mzn$ form an ultra-log-concave sequence as its (even degree) Poincar\'e polynomial is real-rooted by \cite{BK}. For $n$ large enough, the sequence is even 3-ultra-log-concave but not 4-ultra-log-concave by \cite{CK}.
In \S\ref{S2} and \S\ref{S3}, we will review these recent progress.

\medskip

A popular strategy for the geometry and topology of $\mgn$ is to use the forgetful morphism
$$\pi:\mgn\lra \ocM_g,\quad (C,p_1,\cdots,p_n)\mapsto {C}^s$$
where ${C}^s$ is the stabilization of $C$ after forgetting the marked points.
By stratifying $\ocM_g$ and applying the decomposition theorem of Beilinson, Bernstein, Deligne and Gabber to $R\pi_*\QQ$, we may extract cohomological information about $\mgn$ from the base $\ocM_g$ and the fibers of $\pi$.
Over a smooth curve $C\in \cM_g$, the fiber of $\pi$ is the Fulton-MacPherson compactification $C[n]$ of the space
\beq\label{1}
C[n]^\circ =C^n-\bigcup_{i<j} \Delta_{ij}, \quad \Delta_{ij}=\{(p_1,\cdots,p_n)\in C^n\,|\,p_i=p_j\}
\eeq
of $n$ ordered distinct points in $C$ (cf. \cite{FM}). We call $C[n]$ the \emph{FM space} for short.

Although the Betti numbers of $C[n]$ can be explicitly computed by the recipe in \cite{FM} for low $n$, they grow exponentially fast as $n$ grows and it is hard even to find  their effective estimates for large $n$.
Thus it makes sense to consider the asymptotic behavior of the Betti number distribution of $C[n]$.
In this paper, we prove the following theorems which generalize \cite[Theorem 4.1, Corollary 4.2]{CK} to arbitrary smooth projective curves. (See \S\ref{S4}.)

\begin{theorem}\label{2}
For any connected smooth projective curve $C$, the Betti numbers of the FM space $C[n]$ are asymptotically normally distributed with mean $\mu_n=n$ and variance
\beq\label{3} \sigma_n^2=\frac{2(3-e)}{3(e-2)}n+O(1).\eeq
\end{theorem}

\begin{theorem}\label{2_1}
 For any connected smooth projective curve $C$, let $\beta_{n,k}= \dim H^k(C[n])$.
 The subsequences of even Betti numbers $\{\beta_{n,2i}\}$ and odd Betti numbers $\{\beta_{n,2i+1}\}$ are asymptotically 3-ULC but not 4-ULC in the central range $|i - \frac{n}{2}|=O(\sqrt{n})$.
\end{theorem}

Recall that the normal distribution with mean $\mu$ and variance $\sigma^2$ is defined by the Gaussian density function
\beq\label{4} \frac{1}{\sqrt{2\pi}\sigma} \exp\left(-\frac{(x-\mu)^2}{2\sigma^2}\right).\eeq
Typically, normal distributions arise from the sum of identical and  independent random variables and it is rather mysterious why the Betti sequence of $C[n]$ converges to the normal distribution for any curve $C$. Also we do not know the geometric meaning of the variance \eqref{3}.

\medskip

\noindent\textbf{Convention}. All stacks and schemes in this paper are defined over the complex number field $\CC$. All cohomology groups have rational coefficients.

\medskip

\noindent \textbf{Acknowledgement}.
It is our pleasure to thank the organizers of the 2025 Algebraic Geometry in East Asia conference, held in Tainan, for warm hospitality. JC was partially supported by the National Research Foundation of Korea(NRF) grant (RS-2026-25486827).  

\bigskip
\section{Cohomology of $\mzn$}\label{S2}
In this section, we review recent progress on the cohomology of $\mzn$ from \cite{CKL1, CKL2, CKL3} after recalling classical results.

\subsection{Classical results}
In this subsection, we recall fundamental results about $H^*(\mzn)$ from 1990s in the chronological order.

In 1992, Keel discovered an inductive construction of $\mzn$. In fact, $\ocM_{0,n+1}$ is obtained by a sequence of smooth blowups from $\mzn\times \PP^1$. By using the blowup formula, he showed that
$$H^*(\mzn)\cong A^*(\mzn)\cong R^*(\mzn)$$
where $A^*$ and $R^*$ denote the Chow ring and the tautological ring respectively. In particular, all the odd degree cohomology groups vanish. Moreover, the ring is generated by the boundary divisors
\beq\label{10}
D_S=D_{S^c},\quad S\subset \{1,2,\cdots,n\},\ \  2\le |S|\le n-2\eeq
and the relation ideal is generated by
the relations pulled back from $\ocM_{0,4}\cong \PP^1$ and the obvious ones like $D_S\cdot D_{S'}=0$ if $D_S\cap D_{S'}=\emptyset$. See \cite{Keel}.

In 1993, Kapranov proved that there are a birational morphism
$$\rho:\ocM_{0,n+1}\lra \PP^{n-2}$$
and $n$ points $p_1,\cdots,p_n\in \PP^{n-2}$ in general position such that $\rho$ is the composition of blowups, first along the $n$ points $p_i$, second along the proper transforms of the $\binom{n}{2}$ lines $\overline{p_ip_j}$, next along the proper transforms of the $\binom{n}{3}$ planes $\overline{p_ip_jp_k}$ and so on.
The blowups are $S_n$-equivariant where $S_n\subset S_{n+1}$ is the subgroup fixing the last marked point. See \cite{Kapra}. In particular, $\ocM_{0,n+1}$ is $S_n$-rational (i.e. $S_n$-equivariantly birational to $\PP^{n-2}$) but not $S_{n+1}$-rational (simply because $S_{n+1}$ is not a subgroup of $\mathrm{Aut}(\PP^{n-2})$).
In \cite[\S5]{CKL1}, using Kapranov's construction, the authors computed the graded $S_n$-module $H^*(\ocM_{0,n+1})$ as a sum over weighted rooted trees (cf. \cite[Proposition 5.12]{CKL1}).

In 1995, Getzler established an operadic duality between $H^*(\mzn)$ and $H^*(\cM_{0,n})$ where
\[ \cM_{0,n} = \left( (\PP^1)^n-\bigcup_{i<j}\Delta_{ij}\right)/\mathrm{PGL}_2(\CC), \quad   \Delta_{ij}=\{(p_1,\cdots,p_n)\,|\,p_i=p_j\}\]
is the moduli space of $n$ ordered distinct points in $\PP^1$ up to projective equivalence, which is open in $\mzn$.
As a consequence, he showed in \cite{Getzler} that the exponential generating function for the Poincar\'e polynomials of $\mzn$ and that for the virtual Poincar\'e polynomials of $\cM_{0,n}$ are inverse to each other as follows.

Let $b_{n,k}=\dim H^k(\mzn)$ and let
\beq\label{5} P_n(u)=\sum_{k\ge 0} b_{n,k}u^k, \quad \varphi(z,t)=z+\sum_{n\ge 2}\frac{z^n}{n!}P_{n+1}(u)\eeq
be the Poincar\'e polynomial of $\mzn$ and the exponential generating function of the Poincar\'e polynomials, respectively.
Using the smooth surjective morphism $\cM_{0,n+1}\to \cM_{0,n}$ whose fibers are $\PP^1$ minus $n$ points, it is easy to see that the virtual Poincar\'e polynomial $Q_{n+1}(u)$ of $\cM_{0,n+1}$ for $n\ge 3$ is
\[ (u^2-2)(u^2-3)\cdots (u^2-(n-1))=\frac{n!}{u^2(u^2-1)}\binom{u^2}{n}\]
with $Q_3(u)=1$ and thus we obtain
\[\phi(y,u)=y-\sum_{n\ge 2}\frac{y^n}{n!}Q_{n+1}(u)=\frac{u^4y-(1+y)^{u^2}+1}{u^2(u^2-1)}.\]
Getzler's duality tells us that $\phi(\varphi(z,u),u)=z$, so that we get an identity
\beq\label{6}
\frac{u^4\varphi(z,u)-(1+\varphi(z,u))^{u^2}+1}{u^2(u^2-1)}=z \ \ \in \ \QQ[u][\![z]\!].
\eeq
By expanding \eqref{6} as a series in $z$, we can compute the Poincar\'e polynomials $P_n(u)$.
In fact, it is straightforward to check that \eqref{6} is equivalent to the recursive formula
\beq\label{8}
P_{n+1}(u)=(1+u^2)P_n(u)+u^2\sum_{j=3}^{n-1}\binom{n-1}{j-1}P_j(u)P_{n+2-j}(u)
\eeq
which can be also derived from Keel's construction in \cite{Keel} or the new construction in \cite{CK0}.


\subsection{Recent progress}
Recently, based on the advances in moduli theory, new insights have flowed into the study of $\mzn$ which led to a new inductive construction.
 In \cite[Theorem 4.8]{CKL1}, the authors proved a formula
that compares $H^*(\mzn)$ and $H^*(\ocM_{0,n+1})$ as graded $S_n$-modules by using
the new inductive construction of $\mzn$ in \cite{CK0}.
Combining this with the formula \cite[Proposition 5.12]{CKL1} for the $S_n$-module $H^*(\ocM_{0,n+1})$ by Kapranov's construction, the authors obtained a closed formula for the graded $S_n$-representation $H^*(\mzn)$ as follows.

\begin{theorem}\label{9} \cite[Theorem 6.1]{CKL1} As an equality of virtual representations of $S_n$, we have
    \[
 H^{2k}(\mzn)=       \sum_{\substack{i\geq 0 \\ T\in \sT_{n,i}/S_n}}(-1)^{k-i} U_T
 \] \[
 +\frac{1}{2}\sum_{\substack{2\leq h \leq n-2\\ i+j\leq k-1}}\sum_{\substack{T_1\in \sT_{h,i}/S_h\\ T_2\in \sT_{n-h,j}/S_{n-h}\\T_1\neq T_2}}(-1)^{k-i-j}U_{T_1}.U_{T_2}
 +(-1)^k\sum_{\substack{2i\leq k-1\\T\in \sT_{\frac{n}{2},i}/S_{\frac{n}{2}}}}s_{(1,1)}\circ U_T,
\]
where $U_T:=\mathrm{Ind}^{S_n}_{\mathrm{Stab}(T)}e$, and $\sT_{n,i}$ denotes the set of weighted rooted trees with $n$ inputs and weight $i$. By convention, we define $\sT_{\frac{n}{2},i}$ to be empty for odd $n$.
\end{theorem}

It is expected that $H^{2k}(\mzn)$ is a permutation representation (i.e. it has an $S_n$-invariant basis) for all $k$.
As a consequence of Theorem \ref{9}, the following were proved in \cite{CKL1}.
\begin{corollary}
(1) $H^{2i}(\mzn)$ is a permutation representation for $i\le 3$ or $i\ge n-6$.\\
(2) $H^{2i}(\mzn)\oplus H^{2i+2}(\mzn)$ is a permutation representation for any $i$.
\end{corollary}

Setting aside combinatorics, let us see how we obtained the new $S_n$-equivariant inductive construction of $\mzn$ in \cite{CK0}.
We first consider the moduli stack $\mathfrak{M}_{g,n}(\CC/\CC^*,1)$ of all maps $f:C\to \CC/\CC^*$ of degree 1 from a connected projective curve of genus $g$ with at worst nodal singularities and with $n$ smooth distinct marked points to the quotient stack of $\CC$ by the standard action of $\CC^*$.
By definition, $f$ is equivalent to a line bundle $L$ of degree 1 on $C$ with a section $s$.

Next we pick a stability condition that depends on a positive rational number $\delta$, which cuts out an open substack $M^\delta$.
For this, we let $\nu:C\to \bar{C}$ denote the stabilization map so that $\omega_{\bar C}^\log$ is ample. Let $\bar L=\nu_*L$ and $\bar s=\nu_*s$. We say a map $f:C\to \CC/\CC^*$ is $\delta$-stable if for any subsheaf $0\ne \bar L'\subsetneq \bar L$,
$$\frac{\chi(\bar L')+\delta \theta(\bar L',\bar s)}{r(\bar L')}<\frac{\chi(\bar L)+\delta}{r(\bar L)}$$
where $\theta(\bar L',\bar s)=1$ if $\bar s$ factors through $\bar L'$ and $0$ otherwise. Here, $r(\bar L)$ and $r(\bar L')$ denote the leading coefficients of the Hilbert polynomials of $\bar L$ and $\bar L'$ respectively, with respect to the ample line bundle $\omega_{\bar C}^{\mathrm{log}}$.

It was proved in \cite[Theorem 5.10]{CK0} that the moduli stack $M^\delta$ of $\delta$-stable maps admits a projective moduli space for all $\delta$ except for finitely many values called the walls (depending on $g,n$). Moreover, when $g=0$, we proved in \cite[\S7]{CK0} (see also \cite[Theorem 2.8]{CKL1}) that
\begin{enumerate}
    \item  $M^\infty=\ocM_{0.n+1}$;
    \item $M^{0^+}$ is a $\PP^1$-bundle over $\mzn$ for $n$ odd and a blowup of a $\PP^1$-bundle  over $\mzn$ for $n$ even;
    \item the wall crossings are smooth blowups.
\end{enumerate}
As the $\delta$-stability does not distinguish the marked points, the wall crossing blowups are all $S_n$-equivariant and we thus obtain a formula (cf. \cite[Theorem 4.8]{CKL1}) that compares the $S_n$-representations $H^*(\mzn)$ and $H^*(\ocM_{0,n+1})$.

\medskip

In principle, Theorem \ref{9} should answer all the questions about the $S_n$-module $H^*(\mzn)$. However, since the combinatorics of weighted rooted trees is formidable, the closed formula in Theorem \ref{9} is not very effective for actual computation. In \cite{CKL2},
by exploiting the recursive structure of weighted rooted trees,
the authors found effective inductive formulae for the graded $S_n$-modules $H^*(\mzn)$ and $H^*(\ocM_{0,n+1})$.
This then enabled them to find formulae for the Betti numbers of the moduli space $\mzn/S_n$ of stable curves of genus $0$ with $n$ \emph{unordered} marked points and those of $\ocM_{0,n+1}/S_n$. Remarkably, the Betti numbers exhibit polynomial growth as $n$ increases.

Even with the inductive formulae above, the computation is often  difficult due to \emph{plethysm} operations.
In \cite{CKL3}, the authors introduced a map $$\Lambda_n\lra \QQ[q]_{\le n}=\{f\in \QQ[q]\,|\, \deg f\le n\}$$ as the unique homomorphism sending the plethysm to composition, where $\Lambda_n$ denotes the space of symmetric functions of degree $n$. Indeed, a symmetric function $F$ in $x_1,x_2,\dots$ defines a map $$\cS_F: \mathbb{N} \lra \QQ, \quad m \mapsto F(\underbrace{1, \dots, 1}_{m \text{ times}}, 0, 0, \dots).$$
It is easy to see that this is a polynomial function, denoted by $\cS_F(m)$ and called the \emph{characteristic polynomial} of $F$.
It turns out that $\cS_F$ retains key information about $F$ and is much easier to compute.
In \cite{CKL3}, the authors computed the characteristic polynomials of the $S_n$-module $H^*(X)$ and checked their log-concavity for many smooth projective varieties $X$ acted on by $S_n$. They further proved inductive formulae for the characteristic polynomials of the $S_n$-modules $H^*(\mzn)$ and $H^*(\ocM_{0,n+1})$.

\bigskip
\section{Properties of the Betti number sequence of $\mzn$}\label{S3}

In \S\ref{S2}, we have seen that satisfactory answers to the first two items in Question \ref{0} were found for $\mzn$. In this section, we consider the last item in Question \ref{0} for $\mzn$. More specifically, we are interested in  log-concavity properties.

\begin{definition}
A sequence $a_0,a_1,\cdots,a_m\in \RR_{>0}$ of positive real numbers is called $r$-\emph{ultra-log-concave} ($r$-ULC for short) if for any $0<i<m$,
\beq\label{11}
\frac{a_i^2}{\binom{m}{i}^{2r}}\ge \frac{a_{i-1}}{\binom{m}{i-1}^r} \frac{a_{i+1}}{\binom{m}{i+1}^r}.
\eeq
\end{definition}
It is obvious that $(r+1)$-ULC implies $r$-ULC for $r\ge 0$ and $0$-ULC is the usual \emph{log-concavity} while $1$-ULC is the usual ultra-log-concavity.

Based on computations using \eqref{8}, Aluffi-Chen-Marcolli proposed the following.
\begin{conjecture}\label{12} \cite{ACM}
(1) The even degree Betti numbers $\{b_{n,2i}\}_{0\le i\le n-3}$ of $\mzn$ form an ultra-log-concave sequence. \\
(2) The even degree Poincar\'e polynomial $P_n(t)=\sum_i b_{n,2i}t^i$ has only real negative roots.
\end{conjecture}
Note that (1) is an immediate consequence of (2) by Newton's inequalities.

\begin{theorem}\label{14} \cite{BK}
    Conjecture \ref{12} holds for all $n$.
\end{theorem}

A key ingredient in the proof is a bivariate deformation of the Poincar\'e polynomial $P_n(t)$ as follows. Let $F_m(y,t) \in \QQ[y,t]$ be defined inductively by
\beq\label{13} F_{m+1}(y,t)=(my-m+1)F_m(y,t)+y(y+t-1)\partial_yF_m(y,t),\ \  F_1(y,t)=1.\eeq
Let us consider the exponential sum
\beq\label{16} \Phi(x,y,t)=\sum_{m\ge 1} \frac{x^m}{m!}F_m(y,t)\in \QQ[y,t][\![x]\!]\eeq
of all $F_m(y,t)$. Then it is easy to see that \eqref{13} is equivalent to
\beq\label{15}
(1-(y-1)x)\partial_x\Phi-y(y+t-1)\partial_y\Phi=1+\Phi,\quad \Phi(0,y,t)=0.\eeq
We can solve the partial differential equation \eqref{15} by considering the characteristic equations
\beq\label{17}
\frac{dx}{ds}=1-(y-1)x, \ \ \frac{dy}{ds}=-y(y+t-1), \ \ \frac{dt}{ds}=0, \ \  \frac{d\Phi}{ds} = 1+\Phi.
\eeq
We thus obtain
\beq\label{18}
F_m(1,t)=P_{m+1}(t),\quad \partial_yF_m(1,t)=t^{-1}\left( P_{m+2}(t)-P_{m+1}(t)  \right).\eeq
By \eqref{13}, we have
\beq\label{19}
\frac{d}{dy} \left( w_m(y)F_m(y,t)\right)=-\frac{w_m(y)}{y(1-t-y)} F_{m+1}(y,t)\eeq
for each fixed $t<0$ where $$w_m(y):=y^{\frac{m-1}{1-t}}(1-t-y)^{\frac{1-mt}{1-t}}.$$
By induction on $m$ and Rolle's theorem, for fixed $t=0^-$, $F_m(y,t)$ has $(m-2)$ real simple roots in the interval $(0,1)$. Likewise, for fixed $t<\!<0$, $F_m(y,t)$ has $m-2$ real simple roots in the interval $(1,\infty)$. Since the curve $F_m(y,t)=0$ is smooth by the implicit function theorem for $t<0$, we find that it meets the line $y=1$ transversely $m-2$ times. Therefore $P_{m+1}(t)=F_m(1,t)$ has $m-2$ negative real simple roots. The completes a proof of Theorem \ref{14}. See \cite{BK} for details.

\begin{center}
\begin{tikzpicture}[>=latex,thick,scale=0.8]

\tikzset{
    asterisk/.pic={
        \draw[thick] (-0.15,0) -- (0.15,0);
        \draw[thick] (-0.075,-0.13) -- (0.075,0.13);
        \draw[thick] (-0.075,0.13) -- (0.075,-0.13);
    }
}

\draw[->, thick] (-4.5, 0) -- (1.5, 0) node[right] {$t$};
\draw[->, thick] (0, -0.5) -- (0, 4.2) node[right] {$y$};

\draw[thick] (-4.5, 1.5) -- (1.5, 1.5);
\node[above right] at (1, 1.5) {$y=1$};

\draw[thick] (-3.8, 3.3) to[out=0, in=135] (-0.8, 1.5) to[out=-45, in=160] (0, 1.2);
\draw[thick] (-3.8, 2.6) to[out=0, in=135] (-1.8, 1.5) to[out=-45, in=160] (0, 0.8);
\draw[thick] (-3.8, 1.9) to[out=0, in=135] (-2.8, 1.5) to[out=-45, in=160] (0, 0.4);

\pic at (-3.8, 3.3) {asterisk};
\pic at (-3.8, 2.6) {asterisk};
\pic at (-3.8, 1.9) {asterisk};

\pic at (0, 1.2) {asterisk};
\pic at (0, 0.8) {asterisk};
\pic at (0, 0.4) {asterisk};

\filldraw[fill=white, thick] (-0.8, 1.5) circle (3.5pt);
\filldraw[fill=white, thick] (-1.8, 1.5) circle (3.5pt);
\filldraw[fill=white, thick] (-2.8, 1.5) circle (3.5pt);

\end{tikzpicture}
\end{center}
\medskip

\begin{remark} The above proof of Theorem \ref{14} was found with an assistance by an AI research environment called Co-Mathematican by Google DeepMind. The bivariate deformation $F_m(y,t)$ of $P_{m+1}(t)$ in terms of the inductive formula \eqref{13} was introduced by AI, seemingly out of nowhere. Its geometric meaning remained a mystery until \cite{BKs} settled this issue by showing that $F_m(y,t)$ is in fact the \emph{stratified virtual Pioncar\'e polynomial} of $\overline{\mathcal{M}}_{0,m+1}$ multiplied by $y$. See \cite[Corollary 1.6]{BKs}. 
\end{remark}

\medskip
Actually, numerical experiments with \eqref{8} suggest that the even degree Betti sequence $\{b_{n,2i}\}$ of $\mzn$ is not just 1-ULC but also 2-ULC. In fact, $3$-ULC seems to hold as well except for a few exceptions. However, 4-ULC certainly fails. How can we explain these observations?

Let $\chi_n=\sum_i b_{n,2i}$ be the Euler characteristic of $\mzn$. Let
\beq\label{20}
p_{n,i}=\frac{b_{n,2i}}{\chi_n},\quad p_n(t)=\sum_ip_{n,i}t^i.
\eeq
Then we can think of $\{p_{n,i}\}_{0\le i\le n-3}$ as the probability distribution of a random variable with values in $\ZZ_{\ge 0}$.

Let us recall a few basic facts in probability theory.
Given a polynomial $B(t)\in \RR_{\ge 0}[t]$ with $B(1)=1$, its coefficients define a probability distribution of a random variable in $\ZZ_{\ge 0}$ whose mean and variance are given by
\beq\label{21}
\mu=B'(1),\quad \sigma^2=B''(1)+B'(1)-B'(1)^2.\eeq
Given such a polynomial $B(t)$ and $n\ge 1$, obviously $B(t)^n$ also defines a probability distribution with the mean $\mu_n=n\mu$ and the variance $\sigma_n^2=n\sigma^2$.
Now the central limit theorem (in any textbook on probability theory) tells us the following.
\begin{theorem}
    For any $B(t)\in \RR_{\ge 0}[t]$ with $B(1)=1$ and $\sigma^2=B''(1)+B'(1)-B'(1)^2>0$, the sequence of random variables $\xi_n$ defined by the coefficients of $B(t)^n$ is asymptotically normal with mean $n\mu$ and variance $n\sigma^2$, where $\mu$ and $\sigma$ are defined by \eqref{21}. More precisely, the normalized random variable $\frac{\xi_n-n\mu}{\sqrt{n}\sigma}$ converges in distribution to the normal distribution $\cN(0,1)$.
\end{theorem}
Here, a sequence of random variables $\xi_n$ is said to \emph{converge in distribution} to a continuous random variable $\xi$ if
$$\lim_{n\to \infty} \PP(\xi_n\le x) = \PP(\xi\le x),\quad \forall x\in \RR,$$
where $\PP$ denotes the probability measure. We say that a sequence of random variables $\xi_n$ with the mean $\mu_n$ and the variance $\sigma_n^2$ is \emph{asymptotically normal} if the normalized random variable $\frac{\xi_n-\mu_n}{\sigma_n}$ converges in distribution to the normal distribution $\cN(0,1)$.

There are various generalizations of this central limit theorem. A particularly useful generalization for us is the following.
\begin{theorem} \label{22} \cite[Theorem IX.8]{FS}
Let $\xi_n$ be random variables defined by $p_n(t)\in \RR_{\ge 0}[t]$, $p_n(1)=1$. Suppose that there are analytic functions $A(t)$, $B(t)$ with $A(1)=B(1)=1$ such that
\beq\label{23} p_n(t)=A(t)B(t)^n\left(1+O(n^{-1})\right)\eeq
uniformly in a fixed neighborhood of $t=1$. Assume furthermore that the variability condition
\[ B''(1)-B'(1) +B'(1)^2 \ne 0\]
holds.
Then, with
$$\mu_n=n\mu_B+\mu_A+O(n^{-1}),\quad \sigma_n^2=n\sigma_B^2+\sigma_A^2+O(n^{-1}),$$
where $\mu_B, \sigma_B^2$ are defined by \eqref{21} and $\mu_A, \sigma_A^2$ are defined likewise, the random variable $\frac{\xi_n - \mu_n}{\sigma_n}$ converges in distribution to the standard normal distribution $\cN(0,1)$.
\end{theorem}

Let us apply Theorem \ref{22} to $p_n(t)$ defined by \eqref{20}.
By Getzler's identity \eqref{6} with $t=u^2$, one can prove \eqref{23} with
\beq\label{24}
B(t)=\frac{e-2}{t^{\frac{1}{t-1}}-1-t^{-1}}, \ \ A(t)=t^{-\frac{t-2}{2(t-1)}}\left(\frac{t^{\frac{1}{t-1}}-1-t^{-1}}{e(e-2)} \right)^{\frac12}.
\eeq
See \cite[\S3]{CK} for details. By Theorem \ref{22}, we thus obtain the following.
\begin{theorem}\label{25} \cite[Theorem 3.1]{CK}
    The even degree Betti numbers of $\mzn$ are asymptotically normally distributed with the mean and the variance
    $$\mu_n=\frac{n-3}{2}, \quad \sigma_n^2=\frac{3-e}{6(e-2)} n+\frac{4e-11}{12e-24}+O(n^{-1}).$$
\end{theorem}
In fact, a stronger convergence (local limit law) was proved in \cite{CK} and we have the following corollary.
\begin{corollary}\label{cor:ULC}
    The even degree Betti number sequence $\{b_{n,2i}\}$ of $\mzn$ is asymptotically 3-ULC but not 4-ULC in the central range $|i-\frac{n-3}{2}|=O(\sqrt{n})$.
\end{corollary}
The sequence $\binom{n-3}{i}$ of binomial coefficients has variance $\frac{n-3}{4}$ and hence asymptotically
$$\frac{b_{n,2i}}{\binom{n-3}{i}^r}\sim C_n\exp\left( -\frac{(i-\frac{n-3}{2})^2}{2}\left(\frac{1}{\sigma_n^2}-\frac{4r}{n-3} \right) \right)$$
where $C_n$ is a positive constant depending only on $n$.
The corollary follows from the observation that
$$\frac{1}{\sigma_n^2}-\frac{4r}{n-3}\ \ \sim\ \  \frac{4}{n}\left(\frac{6(e-2)}{4(3-e)}-r\right)\ \ >\ \ 0$$
holds for $r=3$ and fails for $r=4$ since $\frac{6(e-2)}{4(3-e)} \approx 3.82$.

Analogously, it was proved in \cite[\S4]{CK} that
the Betti number sequence of the Fulton-MacPherson space $\PP^1[n]$ is also asymptotically normally distributed with the mean $\frac{n}{2}$ and the variance
$$\frac{3-e}{6(e-2)} n+\frac{7e-15}{12e-24}+O(n^{-1}). $$
In the subsequent section, we generalize this to arbitrary smooth projective curves.

We end this section with the following.
\begin{conjecture}
    (1) The even degree Betti sequence $$\{\dim H^{2i}(\mzn/S_n)=\dim H^{2i}(\mzn)^{S_n}\}_{0\le i\le n-3}$$ of $\mzn/S_n$ is log-concave for all $n$.\\
    (2) The even degree Betti sequence of $\mzn/S_n$ is asymptotically normally distributed.
\end{conjecture}
By \cite{CKL2}, it is known that the Betti sequence $\{\dim H^{2i}(\mzn/S_n)\}_{0\le i\le n-3}$ is asymptotically log-concave but not ultra-log-concave.

\bigskip
\section{Fulton-MacPherson space for points on a curve}\label{S4}

In this section, we study the asymptotic properties of the Betti numbers of the Fulton-MacPherson space $C[n]$ for any smooth projective curve $C$. We prove two main results: Theorem \ref{2}, which establishes that the Betti numbers are asymptotically normally distributed, and Theorem \ref{2_1}, which demonstrates that the separated even and odd subsequences of Betti numbers are asymptotically ultra-log-concave.

\medskip

Let $C$ be a smooth projective curve of genus $g$. The Fulton-MacPherson space $C[n]$ for $n$ points on $C$ is a smooth projective variety of dimension $n$ that
compactifies the configuration space $C[n]^\circ$ in \eqref{1} of $n$ ordered distinct points in $C$ whose complement $C[n]-C[n]^\circ$ is a normal crossing divisor.
We can also think of $C[n]$ as the moduli space $\ocM_{g,n}(C,1)$ of stable maps to $C$ of degree $1$.

The Betti numbers of $C[n]$ can be computed as follows. Let
\[\psi(z, u) = 1+ \sum_{n\ge 1} \frac{z^n}{n!}\sum_{k\ge 0} u^k \dim H^k(C[n]).  \]
Then, by \cite[Theorem IV.4.3.1]{Manin}, we have
\beq\label{26}
\psi(z, u) = \left( 1+\varphi (z, u)\right)^{u^2+2g u +1} \eeq
where $\varphi(z,u)$ is defined by \eqref{5}. Since we can compute $\varphi(z,u)$ by \eqref{6} or \eqref{8}, the Betti numbers of $C[n]$ are obtained by expanding \eqref{26}. In fact, \eqref{26} is equivalent to the recursive formula
\beq\label{eq:fmrecur} f_{n+1}(u) = (u^2 + 2gu + 1) f_{n}(u) + u^2 \sum_{j=3}^{n+1} \binom{n}{j-1} P_{j}(u) f_{n+2-j}(u),\eeq
where $f_n(u) = \sum_{k= 0}^{2n} u^k \dim H^k(C[n])$ denotes the Poincar\'e polynomial of $C[n]$. Indeed, by differentiating both \eqref{6} and \eqref{26} with respect to $z$ and combining them, we deduce that $\psi$ satisfies the differential equation
\[ (1+u^2z -u^2\varphi) \psi_z =(1+2gu +u^2 )\psi,  \]
which is equivalent to the recursive formula \eqref{eq:fmrecur}.

To prove Theorem \ref{2} by Theorem \ref{22}, we have to estimate the coefficient of $z^n$ in $\psi(z,u)$. To this end, we first analyze the singularity structure of $\psi$ near $u=1$. (cf. \cite[Section IX.5]{FS}) It turns out that any singular point of $\psi$ in this neighborhood necessarily arises from a singularity of $\varphi$. We apply the inverse function theorem to analyze the latter.

Recall that $\varphi$ can be regarded as the inverse of
\[ z=\frac{u^4\varphi-(1+\varphi)^{u^2}+1}{u^2(u^2-1)}.\]

By the inverse function theorem, singularities of $\varphi$ can only occur when the derivative $ \frac{\partial z}{\partial \varphi}$ vanishes:
\[\frac{\partial z}{\partial \varphi}= \frac{u^4-u^2(1+\varphi)^{u^2-1}}{u^2(u^2-1)}=0.  \]
Thus, we find that the location of the singularity is
\beq\label{eq:sing} z= \rho(u) := u^{\frac{2}{u^2-1}}- \frac{u^2+1}{u^2}, \hspace{1em} \varphi = \lambda(u):= u^\frac{2}{u^2-1}-1. \eeq
The second derivative at the singularity \eqref{eq:sing} is
\beq\label{eq:phi''}\frac{\partial^2 z}{\partial \varphi^2}(z, u)=-u^{\frac{2u^2-4}{u^2-1}}  \eeq
Since this is nonzero, the singularity is of the square-root type (cf. \cite[IX.7.3.]{FS}).

By differentiating \eqref{26} with respect to $z$, we have
\[
\frac{\partial \psi}{\partial z}=(u^2+2gu +1)(1+\varphi(z, u))^{u^2+2g u}\frac{\partial \varphi}{\partial z}(z, u).\]
Hence
\beq\label{7}\frac{\partial z}{\partial \psi}=0 \;\Leftrightarrow \;
\frac{\partial z}{\partial \varphi}(z, u)=0 \;\Leftrightarrow \;
z=\rho(u),\; \varphi=\lambda(u), \; \psi=\gamma(u)\eeq
in a neighborhood of $u=1$, where
$\gamma(u)=u^{\frac{2u^2+4gu+2}{u^2-1}}$.
Moreover, when $\frac{\partial z}{\partial \varphi}=0$, we have $\frac{\partial^2 z}{\partial \psi^2} = \frac{\partial^2 z}{\partial \varphi^2} \left(\frac{\partial \varphi}{\partial \psi}\right)^2$. Hence, by \eqref{eq:phi''},
\[
\frac{\partial^2z}{\partial \psi^2}\bigg|_{\psi=\gamma(u)}=-\frac{1}{(u^2+2gu+1)^2}u^{-\frac{2u^2+8gu+4}{u^2-1}}.
\]
By the Taylor expansion at \eqref{7}, we then have
\[
z-\rho(u) =-\frac{u^{-\frac{2u^2+8gu+4}{u^2-1}}}{2(u^2+2gu+1)^2}
(\psi-\gamma(u))^2+O\left(|\psi-\gamma(u)|^3\right).\]
Inverting this gives
\[
\psi-\gamma(u) = -\sqrt{2} (u^2+2gu+1) u^{\frac{u^2+4gu+2}{u^2-1}}\rho(u)^\frac12
\left(1-\frac{z}{\rho(u)}\right)^\frac12+O\left(\left|1-\frac{z}{\rho(u)}\right|\right).\]
Hence, the coefficient of $z^n$ in $\psi(z,u)$ is
\[ [z^n]\psi(z,u) = \frac{1}{\sqrt{2\pi}}(u^2+2gu+1) u^{\frac{u^2+4gu+2}{u^2-1}}\rho(u)^{\frac12-n}n^{-\frac32} +O(\rho(u)^{-n}n^{-\frac52}). \]
Then we may apply \cite[Theorem VI.4]{FS} 
to find that
\[  \frac{[z^n]\psi(z,u)}{[z^n]\psi(z,1)}=
\frac{(u^2+2gu+1) u^{\frac{u^2+4gu+2}{u^2-1}}}{(2+2g)e^{2g+\frac32}} \left(\frac{\rho(u)}{e-2}\right)^{\frac12 -n}  \left(1+O(n^{-1})\right).\]
Therefore, we may apply Theorem \ref{22} with
\[A(u)=\frac{(u^2+2gu+1) u^{\frac{u^2+4gu+2}{u^2-1}}}{(2+2g)e^{2g+\frac32}}\left(\frac{\rho(u)}{e-2}\right)^{\frac12}, \ \  B(u)=\frac{e-2}{\rho(u)}\]
and conclude that the Betti numbers of $C[n]$ are asymptotically normally distributed. The mean is $n$ and the asymptotic variance is
\[ \sigma_n^2 = \frac{2(3-e)}{3(e-2)}n+ \frac{4g^2-5g-15-e(2g^2-2g-7)}{(3g+3)(e-2)} + O(n^{-1}).\]
This completes our proof of Theorem \ref{2}.

\bigskip

\begin{figure*}[t]
    \makebox[\textwidth][c]{
        \begin{minipage}{1.15\textwidth}
            \centering
            \begin{subfigure}{0.48\linewidth}
                \centering
            \includegraphics[width=\linewidth]{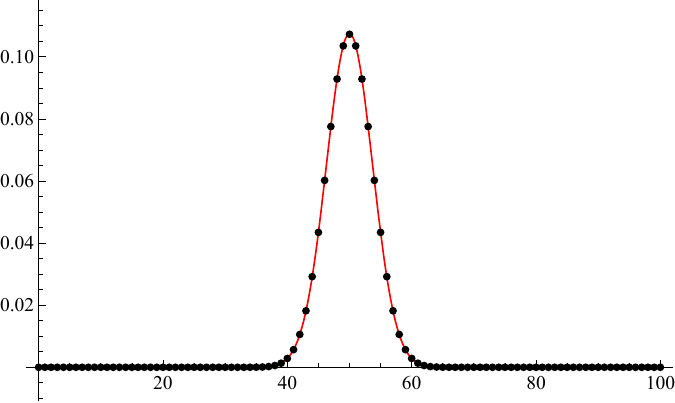}
                \caption{$g=1, \ n=50$ }
                \label{fig:left}
            \end{subfigure}
            \hfill
            \begin{subfigure}{0.48\linewidth}
                \centering
                \includegraphics[width=\linewidth]{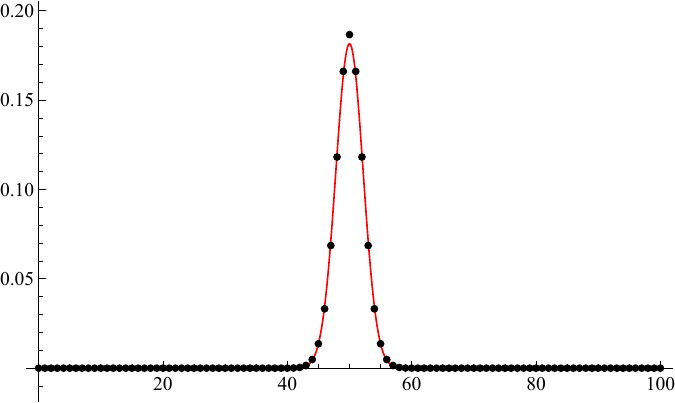}
                \caption{$g=20, \ n=50$}
                \label{fig:right}
            \end{subfigure}
        \end{minipage}
    }
    \caption{Betti number distribution for $C[50]$}
    \label{fig:main}
\end{figure*}

Figure \ref{fig:main} displays the normalized Betti
numbers (black dots) of $C[50]$ when $g=1$ and $g=20$, respectively, together with the probability density function of the corresponding normal distribution (red curve), showing close agreement.

\begin{remark}
When $g=0$, Theorem \ref{2} recovers \cite[Theorem 3.1]{CK}. Note that the mean and the variance are multiplied by 2 and 4, respectively, because we now include odd degree Betti numbers in our consideration. Even though these odd degree Betti numbers vanish, Theorem \ref{22} still applies and the asymptotic normality holds. Recall that the asymptotic normality is the convergence of the distribution functions. However, the local limit theorem fails in this setting when $g=0$, precisely due to the vanishing odd degree Betti numbers. 
\end{remark}

We now turn our attention to local properties: the local limit theorem and the asymptotic log-concavity. Proving asymptotic log-concavity requires highly precise pointwise estimates of the Betti numbers. However, as can be seen from the genus zero example, the local limit theorem does not hold in general and the Betti numbers exhibit a parity-dependent oscillation. Analytically, this is because the modulus $|\rho(u)|$ does not attain a unique minimum on the unit circle $|u|=1$ at $u=1$. Indeed, because $\rho(u)$ is a function of $u^2$, the minimum modulus is attained simultaneously at two antipodal points, $u=\pm 1$. 

To overcome this periodicity, we consider the subsequences of even and odd Betti numbers separately and apply the change of variable $w=u^2$. Define the generating functions for the even and odd Betti numbers as
\[ \psi^\mathrm{even}(z, u) = \frac{\psi(z, u)+ \psi(z, -u)}{2} \and \psi^\mathrm{odd}(z, u) = \frac{\psi(z, u)- \psi(z, -u)}{2u}. \]
By construction, both functions involve only even powers of $u$. We apply the change of variables $w=u^2$:
\[ \Psi^\mathrm{even}(z, w) = \psi^\mathrm{even}(z, \sqrt{w})  \and \Psi^\mathrm{odd}(z, w) =  \psi^\mathrm{odd}(z, \sqrt{w}) .  \]
These are the generating series of the even and odd Betti numbers, respectively. By this change of variables, the dominant singularity of these functions in the $z$-plane is given by \[\tilde{\rho}(w):=\rho(\sqrt{w}) = w^{\frac{1}{w-1}}- \frac{w+1}{w},\]
To analyze these separated subsequences, we employ the following local limit thoerem developed by Hwang.


\begin{theorem}\cite[Theorem 1]{HwangLLT}\label{LLT}
  Let $\{\Omega_n\}$ be a sequence of integer-valued random variables with moment generating functions $M_n(s)=\EE[e^{\Omega_n s}]$. Suppose there exist $\delta>0$ and analytic functions $u(s)$, $v(s)$ such that uniformly for $|s|<\delta$,
  \beq\label{eq:mgf} M_n(s) = e^{n u(s) +v(s)}\left(1+ O(n^{-1}) \right), \quad n\to \infty, \eeq
  satisfying the following conditions:
  \begin{enumerate}
    \item $u(s)$ and $v(s)$ are independent of $n$ and analytic for $|s|\le \delta$ with $u''(0)\ne 0$.
    \item \emph{(Aperiodicity condition)} For a fixed $0<\epsilon<\delta$, there exists $c>0$ such that
    \[ \left| \frac{M_n(r+it)}{M_n(r)}\right|=O(e^{-cn})\]
    uniformly for $|r|\le \delta$ and $\varepsilon \le |t|\le \pi$.
  \end{enumerate}
  Then, letting $\mu_n = n u'(0) +v'(0) $ and $\sigma_n^2 = n u''(0)$, for any integer $m$ in the central region such that $x =\frac{m-\mu_n}{\sigma_n}=O(1)$, the probabilities satisfy asymptotically
  \[ \PP(\Omega_n =m) = \frac{e^{-x^2/2}}{\sqrt{2 \pi \sigma_n^2}}\left( 1 + \frac{\Pi_1(x) }{\sqrt{n}} + \frac{\Pi_2(x) }{n}+ O\left( n^{-\frac{3}{2}}\right) \right),\]
  where $\Pi_1(x) $ and $\Pi_2(x) $ are polynomials.
\end{theorem}

We apply Theorem \ref{LLT} to the even Betti numbers by considering the normalized coefficients as a probability generating function for a random variable. (The argument for odd Betti numbers is parallel.) The probability generating function for the even subsequence is defined by
\[ \tilde{P}_n(w) :=  \frac{ [z^n]\Psi^\mathrm{even}(z, w) }{[z^n]\Psi^\mathrm{even}(z, 1)} \]

To evaluate the numerator, observe that $[z^n]\Psi^\mathrm{even}(z, w) = \frac{1}{2}\big([z^n]\psi(z, \sqrt{w}) + [z^n]\psi(z, -\sqrt{w})\big)$. The asymptotics of the first term are derived above. Similarly as before, we can derive the corresponding asymptotic formula for $\psi(z, -u)$ as
\[ [z^n]\psi(z,-u) = \frac{1}{\sqrt{2\pi}}(u^2-2gu+1) u^{\frac{u^2-4gu+2}{u^2-1}}\rho(u)^{\frac12-n}n^{-\frac32} +O(\rho(u)^{-n}n^{-\frac52}). \]
Using the transfer theorem for algebraic singularities, we combine these to obtain
\[ \tilde{P}_n(w) =\tilde{A}(w) \left(\frac{\tilde{\rho}(w)}{e-2}\right)^{-n}  \left(1+O(n^{-1})\right),\]
where the amplitude function $  \tilde{A}(w)  $ is given by
\[  \tilde{A}(w) = \frac{(w+2g\sqrt{w} +1) w^{\frac{w+4g\sqrt{w} +2}{2(w-1)}}+(w-2g\sqrt{w} +1) w^{\frac{w-4g\sqrt{w} +2}{2(w-1)}}}{(2+2g)e^{2g+\frac32}+(2-2g)e^{-2g+\frac32}} \left(\frac{\tilde{\rho}(w)}{e-2}\right)^{\frac12} \]
Then the moment generating function $M_n(s) =\tilde{P}_n(e^s)$ is of the form \eqref{eq:mgf}, where
\[ u(s) = -\log \frac{\tilde{\rho}(e^s)}{e-2}\quad \text{and} \quad v(s) = \log\tilde{A}(e^s). \]
Since $\tilde{\rho}(w)$ and $\tilde{A}(w)$ are analytic near $w=1$, $u(s)$ and $v(s)$ are analytic near $s=0$. One can compute $u'(0)= \frac12$, $u''(0)= \frac{3-e}{6(e-2)}\ne 0$ and $v'(0)=0$.

It remains to verify the aperiodicity condition. Recall that by the change of variables $w= u^2$, $|\tilde{\rho}(w)|$ attains its unique minimum at $w=1$ along the unit circle $|w|=1$. Consequently,
\beq \label{rhoineq} \min_{w\in \Gamma} |\tilde{\rho}(w) |> \tilde{\rho}(1)= e-2 \eeq
for any closed arc $\Gamma$ on the unit circle $|w|=1$ not containing $w=1$.

Let $c_n(w) = [z^n]\Psi^\mathrm{even}(z, w) $ so that $ \left| \frac{M_n(r+it)}{M_n(r)}\right|=\left| \frac{c_n(e^{r+it})}{c_n(e^r)}\right| $. Similarly as before, one can see for real $r$ in a sufficiently small compact set $[-\delta, \delta]$,
$$c_n(e^r) = \tilde{A}(e^r) \tilde{\rho}(e^r)^{-n} n^{-\frac32} \left(1 + O(n^{-1})\right).$$
Because $\tilde{A}(e^r) $ is continuous and nonvanishing, there is a constant $\tilde{C}>0$ independent of $n$ and $r$ such that $|c_n(e^r)| \ge \tilde{C} \tilde{\rho}(e^r)^{-n} n^{-\frac32}$, uniformly for $|r|\le\delta$ and large $n$.

For $w= e^{r+it}$, the radius of convergence of the power series $\Psi^\mathrm{even}(z, w) $ in $z$ is exactly $ |\tilde{\rho} (e^{r+it}) |$. By \eqref{rhoineq} and the continuity of $|\tilde{\rho}(w)|$, we can choose $R$ independent of $t$ such that
$$\tilde{\rho}(e^r) < R < \min_{\varepsilon \le |t| \le \pi} |\tilde{\rho}(e^{r+it})|.$$
Since the radius of convergence is strictly greater than $R$, by the Cauchy-Hadamard theorem, $|c_n(e^{r+it})|= O(R^{-n})$. Taking the ratio, we have
\[\left| \frac{M_n(r+it)}{M_n(r)}\right|\le \frac{O(R^{-n})}{\tilde{C}\tilde{\rho}(e^r)^{-n} n^{-\frac32}} =O(e^{-cn}) \]
for some constant $c>0$, because $\tilde{\rho}(e^r) < R $. Therefore, the conditions of Theorem \ref{LLT} are satisfied.

\medskip 

We are now ready to prove Theorem \ref{2_1}.
\begin{proof}[Proof of Theorem \ref{2_1}]
  Let $\tilde{p}_{n,i}= [w^i]\tilde{P}_n(w) $. Note that $\tilde{p}_{n,i}= \beta_{n,2i} / \sum_j \beta_{n,2j}$ is the normalized probability corresponding to the even Betti numbers. Hence, to show the sequence $\{\tilde{p}_{n,i}\}$ is asymptotically $r$-ULC, it suffices to show that the second discrete difference of the logarithmic ratio
  \[\Delta^2_i \ln \left(\frac{\tilde{p}_{n,i}}{\binom{n}{i}^r}\right) =\ln \left(\frac{\tilde{p}_{n,i-1}}{\binom{n}{i-1}^r}\right)-2 \ln \left(\frac{\tilde{p}_{n,i}}{\binom{n}{i}^r}\right)+\ln \left(\frac{\tilde{p}_{n,i+1}}{\binom{n}{i+1}^r}\right) \]
  is asymptotically strictly negative.
  By Theorem \ref{LLT}, for $i$ in the central region $|i - \frac{n}{2}|=O(\sigma_n)=O(\sqrt{n})$,
  \[ \Delta^2_i \ln \tilde{p}_{n,i} = - \frac{1}{\sigma_n^2}+ O\left( n^{-\frac{3}{2}}\right), \]
  and since $\Delta^2_i \ln {\binom{n}{i}^r}= - \frac{4r}{n}+ O\left( n^{-2}\right)$, the required inequality evaluates to
  \[\Delta^2_i \ln \left(\frac{\tilde{p}_{n,i}}{\binom{n}{i}^r}\right) = - \frac{1}{\sigma_n^2}+\frac{4r}{n}+O\left( n^{-\frac{3}{2}}\right)<0.\]
  Because the error term $O\left( n^{-\frac{3}{2}}\right)$ decays asymptotically faster than the main term of order $n^{-1}$, the inequality holds for sufficiently large $n$ if and only if $r< \frac{n}{4\sigma_n^2}$. Evaluating the threshold ratio, we get
  \[ \frac{n}{4\sigma_n^2} = \frac{1}{4 u''(0)}= \frac{6(e-2)}{4 (3-e)}\approx 3.82. \]
  Therefore the subsequence is asymptotically 3-ULC but not 4-ULC. The proof for the subsequence of the odd Betti numbers is parallel.
\end{proof}

\begin{remark}
  We remark that the local limit theorem fails for the complete unseparated sequence of Betti numbers $\{\beta_{n,k}\}$. Even though the even and odd subsequences share the same mean and variance, they have distinct amplitude functions $\tilde{A}(w)$. Consequently, the pointwise asymptotic formula for $\beta_{n,k}$ exhibits a parity-dependent oscillation, switching between two distinct amplitudes. Hence we do not expect asymptotic log-concavity to hold for the unseparated sequence of Betti numbers.
\end{remark}


\bigskip

\bibliographystyle{amsplain}

\end{document}